\documentclass[12pt]{article}

\usepackage{amsmath,amsthm}
\usepackage{amssymb}
\usepackage{hyperref}
\usepackage{lscape}

\def\fl#1{\left\lfloor#1\right\rfloor}

\newtheorem{theorem}{Theorem}
\newtheorem{Prop}{Proposition}

\allowdisplaybreaks

\begin{document}

\title{A Short Note On Laguerre Polynomials}

\author{
Praveen Agarwal\\
\small Department of Mathematics\\
\small Anand International College of Engineering\\
\small Jaipur 303 012 India\\
\small \texttt{praveen.agarwal@anandice.ac.in}\\\\
Takao Komatsu\\
\small Department of Mathematics, School of Science\\
\small Zhejiang Sci-Tech University\\
\small Hangzhou 310018 China\\
\small \texttt{komatsu@zstu.edu.cn}
}

\date{
\small MR Subject Classifications: 11B68, 11B37, 11C20, 15A15, 33C20.
}

\maketitle

\begin{abstract}
Motivated by the work of Prajapati \emph{et al.} \cite{PAA}, here we study some explicit form of the generalized Laguerre polynomials $L_{\fl{\frac{n}{q}}}^{(\alpha,\beta)}(z)$, when $q=1$. 
\end{abstract}

\section{Introduction}

In \cite{PAA}, for $\alpha,\beta\in\mathbb C$ with $\mathfrak{Re}(\beta)>-1$ and $m,q\in\mathbb N:=\{1,2,\dots\}$ the polynomial
\begin{equation}
L_{\fl{\frac{n}{q}}}^{(\alpha,\beta)}(z)=\frac{\Gamma(\alpha n+\beta+1)}{\Gamma(n+1)}\sum_{k=0}^{\fl{\frac{n}{q}}}\frac{(-n)^{(q k)}}{\Gamma(\alpha k+\beta+1)}\frac{z^k}{k!}
\label{pol:paa}
\end{equation}
is introduced and studied, where $\fl{n/q}$ denotes the integral part of $n/q$ and $(\gamma)^{(n)}=\Gamma(\gamma+n)/\Gamma(\gamma)$ ($n\ge 0$) denotes the rising factorial.
When $q=\alpha=1$,  $L_n^{(\beta)}(z)=L_n^{(1,\beta)}(z)$ is the generalized Laguerre polynomial (\cite{Rainville}).

The closed form of $L_{\fl{\frac{n}{q}}}^{(\alpha,\beta)}(z)$ is also given as follows.

\begin{Prop}
If $\alpha$ is a positive integer,
$$
L_{\fl{\frac{n}{q}}}^{(\alpha,\beta)}(z)=\sum_{j=0}^{\fl{\frac{n}{q}}}(-1)^j\binom{\alpha n+\beta}{\alpha(n-j)}\frac{\bigl(\alpha(n-j)\bigr)!}{(n-q j)!}\frac{z^j}{j!}\,.
$$
\label{th:closedform}
\end{Prop}

\noindent
{\it Remark.}
When $q=\alpha=1$, we have the closed form of the generalized Laguerre polynomials:
$$
L_n^{(\beta)}(z)=\sum_{j=0}^n(-1)^j\binom{n+\beta}{n-j}\frac{z^j}{j!}
$$
(\cite[(112.3)]{Rainville}).

The recurrence relation for $q=1$ can be given as follows.

\begin{theorem}
For $n\ge 3$,
$$
L_n^{(\alpha,\beta)}(z)=-\frac{z}{n}L_{n-1}^{(\alpha,\beta)}(z)+\sum_{m=1}^n(-1)^{m-1}\bigl(\alpha(n-1)+\beta\bigr)_{(m-1)\alpha}C_{n,m}L_{n-m}^{(\alpha,\beta)}(z)\,,
$$
where
$$
C_{n,m}=C_{n,m}(\alpha,\beta):=\frac{1}{m!n}\sum_{j=0}^m(-1)^j\binom{m}{j}(n-j)\bigl(\alpha(n-j)+\beta\bigr)_{\alpha}
$$
and $(\gamma)_n:=\Gamma(\gamma+1)/\Gamma(\gamma-n+1)$ ($n\ge 0$) denotes the falling factorial.
\label{th:rec}
\end{theorem}

\noindent
{\it Remark.}
When $\alpha=1$, $C_{n,1}=\frac{2 n-1+\beta}{n}$, $C_{n,2}=\frac{n-1+\beta}{n}$ and $C_{n,m}=0$ ($m\ge 3$). Then, as well-known, the three-term recurrence relation for the generalized Laguerre polynomials is given by
$$
L_n^{(1,\beta)}(z)=\frac{(2 n-1+\beta-z)L_{n-1}^{(1,\beta)}(z)-(n-1+\beta)L_{n-2}^{(1,\beta)}(z)}{n}
$$
(\cite[114.4]{Rainville}).

\begin{proof}[Proof of Theorem \ref{th:rec}.]
By Theorem \ref{th:closedform},
\begin{align*}
&-\frac{z}{n}L_{n-1}^{(\alpha,\beta)}(z)+\sum_{m=1}^n(-1)^{m-1}\bigl(\alpha(n-1)+\beta\bigr)_{(m-1)\alpha}C_{n,m}L_{n-m}^{(\alpha,\beta)}(z)\\
&=-\frac{z}{n}\sum_{j=0}^{n-1}(-1)^j\frac{\Gamma(\alpha(n-1)+\beta+1)}{\Gamma(\alpha j+\beta+1)\Gamma(n-j)}\frac{z^j}{j!}\\
&\quad+\sum_{m=1}^n(-1)^{m-1}\frac{\Gamma(\alpha(n-1)+\beta+1)}{\alpha(n-m)+\beta+1}C_{n,m}\\
&\qquad \times\sum_{j=0}^{n-m}(-1)^j\frac{\Gamma(\alpha(n-m)+\beta+1)}{\Gamma(\alpha j+\beta+1)\Gamma(n-j-m+1)}\frac{z^j}{j!}\\
&=\frac{1}{n}\sum_{j=1}^n(-1)^j\frac{\Gamma(\alpha(n-1)+\beta+1)}{\Gamma(\alpha(j-1)+\beta+1)\Gamma(n-j+1)}\frac{z^j}{(j-1)!}\\
&\quad +\sum_{j=0}^{n-1}\frac{(-1)^j z^j}{j!}\sum_{m=1}^{n-j}\frac{(-1)^{m-1}\Gamma(\alpha(n-1)+\beta+1)C_{n,m}}{\Gamma(\alpha j+\beta+1)\Gamma(n-j-m+1)}\,.
\end{align*}
Since
\begin{align*}
&\sum_{m=1}^{n-j}\frac{(-1)^{m-1}C_{n,m}}{\Gamma(n-j-m+1)}\\
&=\frac{\Gamma(\alpha n+\beta+1)}{\Gamma(\alpha(n-1)+\beta+1)\Gamma(n-j+1)}\\
&\quad+\frac{1}{n}\sum_{l=0}^{n-j}(-1)^l(n-l)\frac{\Gamma(\alpha(n-l)+\beta+1)}{\Gamma(\alpha(n-l-1)+\beta+1)}\sum_{m=l}^{n-j}\binom{m}{l}\frac{(-1)^{m-1}}{m!\Gamma(n-j-m+1)}\\
&=\frac{\Gamma(\alpha n+\beta+1)}{\Gamma(\alpha(n-1)+\beta+1)\Gamma(n-j+1)}
+\frac{(-1)^{n}j\Gamma(\alpha j+\beta+1)}{n\Gamma(\alpha(j-1)+\beta+1)(n-j)!}\,,
\end{align*}
we have
\begin{align*}
&-\frac{z}{n}L_{n-1}^{(\alpha,\beta)}(z)+\sum_{m=1}^n(-1)^{m-1}\bigl(\alpha(n-1)+\beta\bigr)_{(m-1)\alpha}C_{n,m}L_{n-m}^{(\alpha,\beta)}(z)\\
&=\sum_{j=1}^n\frac{(-1)^j\Gamma(\alpha(n-1)+\beta+1)}{\Gamma(\alpha(j-1)+\beta+1)}\frac{z^j}{n(n-j)!(j-1)!}\\
&\quad-\sum_{j=1}^{n-1}\frac{(-1)^j\Gamma(\alpha(n-1)+\beta+1)}{\Gamma(\alpha(j-1)+\beta+1)}\frac{z^j}{n(n-j)!(j-1)!}\\
&\quad+\sum_{j=0}^{n-1}\frac{(-1)^j\Gamma(\alpha n+\beta+1)}{\Gamma(\alpha j+\beta+1)}\frac{z^j}{n(n-j)!j!}\\
&=\sum_{j=0}^n\frac{(-1)^j\Gamma(\alpha n+\beta+1)}{\Gamma(\alpha j+\beta+1)}\frac{z^j}{(n-j)!j!}\\
&=L_{n}^{(\alpha,\beta)}(z)\,.
\end{align*}
\end{proof}

We have an expression of $L_n^{(\alpha,\beta)}(z)$ in terms of determinant.

\begin{theorem}
For $n\ge 1$,
$$
L_n^{(\alpha,\beta)}(z)=\left|\begin{array}{ccccc}
B_{n,1}-\frac{z}{n}&1&0&\cdots&0\\
B_{n,2}&B_{n-1,1}-\frac{z}{n-1}&1&&\vdots\\
B_{n,3}&B_{n-1,2}&&\ddots&0\\
\vdots&\vdots&&B_{2,1}-\frac{z}{2}&1\\
B_{n,n}&B_{n-1,n-1}&\cdots&B_{2,2}&B_{1,1}-z
\end{array}\right|\,,
$$
where $B_{n,m}=\left(\alpha(n-1)+\beta\right)_{(m-1)\alpha}C_{n,m}$ ($n\ge m\ge 1$).
In particular, when $\alpha=1$,
$$
L_n^{(1,\beta)}(z)=\left|\begin{array}{ccccc}
B_{n,1}-\frac{z}{n}&1&0&\cdots&0\\
B_{n,2}&B_{n-1,1}-\frac{z}{n-1}&1&&\vdots\\
0&B_{n-1,2}&&\ddots&0\\
\vdots&\ddots&&B_{2,1}-\frac{z}{2}&1\\
0&\cdots&0&B_{2,2}&B_{1,1}-z
\end{array}\right|\,.
$$
\label{ex:det}
\end{theorem}
\begin{proof}
The proof is done by induction on $n$.  For $n=1$, it is easy to see that
$$
L_1=-z L_0+B_{1,1}L_0=B_{1,1}-z\,.
$$
Assume that the result holds up to $n-1$. Then, by repeatedly expanding the determinant along the first row, the right-hand side is equal to
\begin{align*}
&\left(B_{n,1}-\frac{z}{n}\right)\left|\begin{array}{ccccc}
B_{n-1,1}-\frac{z}{n-1}&1&0&\cdots&0\\
B_{n-1,2}&B_{n-2,1}-\frac{z}{n-2}&1&&\vdots\\
B_{n-1,3}&B_{n-2,2}&&\ddots&0\\
\vdots&\vdots&&B_{2,1}-\frac{z}{2}&1\\
B_{n-1,n-1}&B_{n-2,n-2}&\cdots&B_{2,2}&B_{1,1}-z
\end{array}\right|\\
&\quad +\left|\begin{array}{ccccc}
B_{n,2}&1&0&\cdots&0\\
B_{n,3}&B_{n-2,1}-\frac{z}{n-2}&1&&\vdots\\
B_{n,4}&B_{n-2,2}&\ddots1&0\\
\vdots&\vdots&&B_{2,1}-\frac{z}{2}&1\\
B_{n,n}&B_{n-2,n-2}&\cdots&B_{2,2}&B_{1,1}-z
\end{array}\right|\\
&=\left(B_{n,1}-\frac{z}{n}\right)L_{n-1}-B_{n,2}L_{n-1}\\
&\quad +\left|\begin{array}{ccccc}
B_{n,3}&1&0&\cdots&0\\
B_{n,4}&B_{n-3,1}-\frac{z}{n-3}&1&&\vdots\\
B_{n,5}&B_{n-3,2}&&\ddots&0\\
\vdots&\vdots&&B_{2,1}-\frac{z}{2}&1\\
B_{n,n}&B_{n-3,n-3}&\cdots&B_{2,2}&B_{1,1}-z
\end{array}\right|\\
&=\left(B_{n,1}-\frac{z}{n}\right)L_{n-1}-B_{n,2}L_{n-1}+B_{n,3}L_{n-3}-\cdots+
\left|\begin{array}{cc}
B_{n,n-1}&1\\
B_{n,n}&B_n-z
\end{array}\right|\\
&=\left(B_{n,1}-\frac{z}{n}\right)L_{n-1}-B_{n,2}L_{n-1}+B_{n,3}L_{n-3}-\cdots\\
&\quad +(-1)^n B_{n,n-1}L_1+(-1)^{n+1}B_{n,n}\\
&=L_n\,.
\end{align*}
The last part depends on the recurrence relation in Theorem \ref{th:rec}.

When $\alpha=1$, by $C_{n,m}=0$ ($m\ge 3$) we know that $B_{n,m}=0$ ($n\ge 3$).
\end{proof}

From Theorem \ref{ex:det} with the recurrence relation in Theorem \ref{th:rec}, we have another explicit form of the polynomial $L_n^{(\alpha,\beta)}(z)$.

\begin{theorem}
For $n\ge 1$,
$$
L_n^{(\alpha,\beta)}(z)=\sum_{k=1}^n(-1)^{n-k}\sum_{t_1+\cdots+t_k=n\atop t_1,\dots,t_k\ge 1}a_0^{(t_1)}a_{t_1}^{(t_2)}a_{t_1+t_2}^{(t_3)}\cdots a_{t_1+\cdots+t_{k-1}}^{(t_k)}\,,
$$
where
$$
a_j^{(\ell)}=B_{\ell+j,\ell}-\frac{z}{j+1}\delta_{\ell,1}\quad(1\le\ell\le n-j;~j=0,1,\dots,n-2)
$$
and $\delta_{\ell,1}=1$ ($\ell=1$); $0$ ($\ell\ne 1$).
\label{th:ex2}
\end{theorem}
\begin{proof}
In general, we shall prove that
\begin{align}
L_n:&=\left|\begin{array}{cccccc}
a_0^{(1)}&1&0&&&\\
a_0^{(2)}&a_1^{(1)}&1&&&\\
a_0^{(3)}&a_1^{(2)}&a_2^{(1)}&\ddots&&\\
\vdots&\vdots&&&1&0\\
a_0^{(n-1)}&a_1^{(n-2)}&&\ddots&a_{n-2}^{(1)}&1\\
a_0^{(n)}&a_1^{(n-1)}&a_2^{(n-2)}&\cdots&a_{n-2}^{(2)}&a_{n-1}^{(1)}
\end{array}\right|\notag\\
&=\sum_{k=1}^n(-1)^{n-k}\sum_{t_1+\cdots+t_k=n\atop t_1,\dots,t_k\ge 1}a_0^{(t_1)}a_{t_1}^{(t_2)}a_{t_1+t_2}^{(t_3)}\cdots a_{t_1+\cdots+t_{k-1}}^{(t_k)}\quad(n\ge 1)
\label{trudi}
\end{align}
with $L_0=1$.
The assertion (\ref{trudi}) is clear for $n=1,2$ because $L_1=a_0^{(1)}$ and $L_2=-a_0^{(2)}+a_0^{(1)}a_1^{(1)}$. Assume that the assertion (\ref{trudi}) is valid up to $n-1$.
In a similar method to Theorem \ref{ex:det} together with Theorem \ref{th:rec}, expanding the right-hand side of the determinant of (\ref{trudi}) along the last column repeatedly, we have
\begin{align*}
L_n&=a_{n-1}^{(1)}L_{n-1}-a_{n-2}^{(2)}L_{n-2}+a_{n-3}^{(3)}L_{n-3}-\cdots+(-1)^n a_1^{(n-1)}+(-1)^{n+1}a_0^{(n)}L_0\\
&=\sum_{k=1}^{n-1}(-1)^{n-k-1}\sum_{t_1+\cdots+t_k=n-1}a_0^{(t_1)}a_{t_1}^{(t_2)}a_{t_1+t_2}^{(t_3)}\cdots a_{t_1+\cdots+t_{k-1}}^{(t_k)}a_{n-1}^{(1)}\\
&\quad -\sum_{k=1}^{n-2}(-1)^{n-k}\sum_{t_1+\cdots+t_k=n-2}a_0^{(t_1)}a_{t_1}^{(t_2)}a_{t_1+t_2}^{(t_3)}\cdots a_{t_1+\cdots+t_{k-1}}^{(t_k)}a_{n-2}^{(2)}\\
&\quad +\sum_{k=1}^{n-1}(-1)^{n-k-1}\sum_{t_1+\cdots+t_k=n-3}a_0^{(t_1)}a_{t_1}^{(t_2)}a_{t_1+t_2}^{(t_3)}\cdots a_{t_1+\cdots+t_{k-1}}^{(t_k)}a_{n-3}^{(3)}\\
&-\cdots\\
&+(-1)^n\sum_{k=1}^{1}(-1)^{k-1}\sum_{t_1+\cdots+t_k=1}a_0^{(t_1)}a_{t_1}^{(t_2)}a_{t_1+t_2}^{(t_3)}\cdots a_{t_1+\cdots+t_{k-1}}^{(t_k)}a_{1}^{(n-1)}\\
&+(-1)^{n+1}a_0^{(n)}\\
&=\sum_{k=1}^{n-1}(-1)^{n-k-1}\sum_{t_1+\cdots+t_k=n-1}a_0^{(t_1)}a_{t_1}^{(t_2)}a_{t_1+t_2}^{(t_3)}\cdots a_{t_1+\cdots+t_{k-1}}^{(t_k)}a_{t_1+\cdots+t_k}^{(t_{k+1})}\\
&\qquad\qquad\qquad(t_1+\cdots+t_k=n-1,~t_{k+1}=1)\\
&\quad -\sum_{k=1}^{n-2}(-1)^{n-k}\sum_{t_1+\cdots+t_k=n-2}a_0^{(t_1)}a_{t_1}^{(t_2)}a_{t_1+t_2}^{(t_3)}\cdots a_{t_1+\cdots+t_{k-1}}^{(t_k)}a_{t_1+\cdots+t_k}^{(t_{k+1})}\\
&\qquad\qquad\qquad(t_1+\cdots+t_k=n-2,~t_{k+1}=2)\\
&\quad +\sum_{k=1}^{n-1}(-1)^{n-k-1}\sum_{t_1+\cdots+t_k=n-3}a_0^{(t_1)}a_{t_1}^{(t_2)}a_{t_1+t_2}^{(t_3)}\cdots a_{t_1+\cdots+t_{k-1}}^{(t_k)}a_{t_1+\cdots+t_k}^{(t_{k+1})}\\
&\qquad\qquad\qquad(t_1+\cdots+t_k=n-3,~t_{k+1}=3)\\
&-\cdots\\
&+(-1)^n\sum_{k=1}^{1}(-1)^{k-1}\sum_{t_1+\cdots+t_k=1}a_0^{(t_1)}a_{t_1}^{(t_2)}a_{t_1+t_2}^{(t_3)}\cdots a_{t_1+\cdots+t_{k-1}}^{(t_k)}a_{t_1+\cdots+t_k}^{(t_{k+1})}\\
&\qquad\qquad\qquad(t_1+\cdots+t_k=1,~t_{k+1}=n-1)\\
&+(-1)^{n+1}a_0^{(n)}\\
&=\sum_{k=1}^n(-1)^{n-k}\sum_{t_1+\cdots+t_k=n\atop t_1,\dots,t_k\ge 1}a_0^{(t_1)}a_{t_1}^{(t_2)}a_{t_1+t_2}^{(t_3)}\cdots a_{t_1+\cdots+t_{k-1}}^{(t_k)}\,.\\
&\qquad\qquad\qquad(k\to k-1)
\end{align*}
\end{proof}

\end{document}